\documentclass[12pt]{article}
\usepackage{amsbsy}
\usepackage{amssymb}
\usepackage{amsfonts}
\usepackage{amsmath,amscd}
\usepackage{amsopn}
\usepackage{amsthm}
\usepackage[applemac]{inputenc}
\usepackage[T1]{fontenc}
\usepackage[square]{natbib}

\newtheorem{theorem}{Theorem}
\newtheorem{corollary}{Corollary}
\newtheorem{proposition}{Proposition}
\newtheorem{lemma}{Lemma}
\theoremstyle{definition}

\theoremstyle{definition}

\title{Strong Law of Large Numbers for Fragmentation Processes}
\author{S. C. Harris\thanks{Department of Mathematical Sciences, University of Bath, Claverton Down, Bath, BA2 7AY, U.K.}, \, R. Knobloch$^*$ and A. E. Kyprianou$^*$\thanks{Corresponding author}}

\begin{document}

\maketitle
 \begin{abstract}
\noindent In the spirit of a classical results for Crump-Mode-Jagers processes, we prove a strong law of large numbers for homogenous fragmentation processes. Specifically, for self-similar fragmentation processes, including homogenous processes, we prove the almost sure convergence of an empirical measure associated with the stopping line corresponding to first fragments of size strictly smaller than $\eta$ for $1\geq \eta >0$.



\end{abstract}

\section{Fragmentation Processes}

Fragmentation processes have been the subject of an increasing body of literature and the culmination of this activity has recently been summarised, for example, in the recent book of Bertoin \cite{bertfragbook}. Some of the  mathematical roots of fragmentation processes lay with older families of spatial branching processes that have also seen periods of extensive interest such as branching random walks and Crump-Mode-Jagers processes. Irrespective of modern or classical perspectives, such models exemplify the phenomena of random splitting according to systematic  rules and, as stochastic processes, they may be seen as modelling  the growth of special types of multi-particle systems. In such cases where there is a random accumulation of many particles, from a mathematical perspective, it is natural to 
 look for classical behaviour such as large deviations, central limit theorems, local limit theorems and strong laws of large numbers. Many studies have already been carried out in this spirit, see for example \cite{bertfragbook, bertoinselfsimilar, bertoinasym, bertoinsmall, bertoingnedin, bertoinrouault1, bertoinrouault2, berest, krell}. In this article our aim is to contribute to this family of literature by proving a strong law of large numbers for fragmentation processes. To state more clearly the main result, we shall first devote some time defining several of the quantities involved. 

We are interested in the Markov process  $\mathbf{X}:=\{\mathbf{X}(t) : t\geq 0\}$ where $\mathbf{X}(t) =(X_1(t), X_2(t), \cdots)$ and  takes values in 
\[
\mathcal{S}: = \left\{\mathbf{s}=(s_1, s_2, \cdots) : s\geq  s_2 \geq \cdots \geq 0 , \, \sum_{i=1}^\infty s_i \leq 1\right\},
\]
that is to say, the infinite simplex of decreasing numerical sequences with sum bounded from above by 1. Its probabilities will be denoted by $\{\mathbb{P}_{\bf s}: {\bf s}\in\mathcal{S}\}$ and, for $s\in(0,1]$, we shall reserve the special notation $\mathbb{P}_s$ as short hand for $\mathbb{P}_{(s,0,\cdots)}$ and in particular write $\mathbb{P}$ for $\mathbb{P}_1$.   The process $\mathbf{X}$ possesses the fragmentation property, to be understood as follows. Given that $\mathbf{X}(t)  = (s_1,s_2,\cdots)$ where $t\geq 0$, then for $u>0$, $\mathbf{X}(t+u)$ has the same law as the variable obtained by ranking in decreasing order the sequences $\mathbf{X}^{(1)}(u), \mathbf{X}^{(2)}(u),\cdots$ where the latter are independent, random mass partitions with values in $\mathcal{S}$ having the same distribution as $\mathbf{X}(u)$ under $\mathbb{P}_{s_1}$, $\mathbb{P}_{s_2}$, $\cdots$ respectively. The process $\mathbf{X}$, henceforth called a mass fragmentation process, is said to be self-similar with index $\alpha\in\mathbb{R}$ if further, for every $r>0$, the law of $\{r\mathbf{X}(r^\alpha t): t\geq 0\}$ under $\mathbb{P}$ is $\mathbb{P}_r$.  In the special case that $\alpha=0$, we call $\mathbf{X}$ a homogenous fragmentation process. The monograph of Bertoin \cite{bertfragbook} gives a very precise and complete account of the existence and characterization of such processes. However, for later convenience, we shall provide here some additional but modest insight into the stochastic structure of homogenous processes.

It is known that homogenous fragmentation processes can be characterized by a dislocation measure $\nu$ on $\mathcal{S}$ such that $\nu(\{(1,0,\cdots)\})=0$ and 
\[
\int_{\mathcal{S}}(1-s_1)\nu(d{\bf s})<\infty.
\]
One may also include the possibility of continuous erosion of mass, however, this feature will be excluded in this article. In the case that 
\begin{equation}
\nu\left(\sum_{i=1}^\infty s_i <1\right)=0
\label{conservative}
\end{equation}
we say that the fragmentation process is {\it conservative} and otherwise {\it dissipative}.

Roughly speaking, the dislocation measure specifies the rate at which blocks split so that a block of mass $x$ dislocates into a mass partition $x{\bf s}$, where $\mathbf{s}\in\mathcal{S}$, at rate $\nu(d{\bf s})$. To be more precise, consider a Poisson point process $\{({\bf s}(t), k(t)) : t\geq 0\}$ with values in $\mathcal{S}\times \mathbb{N}$ and intensity measure $\nu\otimes\sharp$ where $\sharp$ denotes the counting measure on $\mathbb{N}=\{1,2,\cdots\}$.  Then the homogenous fragmentation process associated with $\nu$  changes state at all times $t\geq0$ for which an atom $({\bf s}(t), k(t))$ occurs in $\mathcal{S}\backslash\{(1,0,\cdots)\}\times \mathbb{N}$.  At such a time $t$, the sequence $\mathbf{X}(t)$ is obtained from $\mathbf{X}(t-)$ by replacing its $k(t)$-th term, $X_{k(t)}(t-)$, with the sequence $X_{k(t)}(t-){\bf s}(t)$ and ranking the terms in decreasing order. When $\nu(\mathcal{S})<\infty$  there is finite activity over finite intervals of time in the underlying Poisson point process. Otherwise said, individual blocks remain unchanged for exponential periods of time with parameter $\nu(\mathcal{S})$. In this case, by taking the negative logarithm of fragment sizes, the fragmentation process is akin to a Markovian (or continuous time) branching random walk; see for example Biggins \cite{big92} or Uchyama \cite{Uchyama}. One may also think of a fragmentation process in this setting as closely related to a Crump-Mode-Jagers process where the negative logarithm of fragment sizes plays the role of birth times; see for examples Jagers \cite{jagers}. The case that $\nu(\mathcal{S})=\infty$ is the more interesting case  in the sense that there are a countable but infinite number of dislocations over any finite time horizon and therefore, mathematically speaking, many results need to be handled differently to the case of branching random walks or Crump-Mode-Jagers processes.

The above construction focuses on mass partitions, however, it can also be seen as the natural consequence of a more elaborate Poissonian procedure concerning the partition of the natural numbers. Let $\mathcal{P}$ be the space of partitions of the natural numbers. Here a partition of $\mathbb{N}$ is a sequence $\pi=(\pi_1, \pi_2, \cdots)$ of disjoint sets, called blocks, such that $\bigcup_i \pi_i = \mathbb{N}$. The blocks of a partition are enumerated in the increasing order of their least element; that is to say $\min \pi_i\leq \min\pi_j$ when $i\leq j$ (with the convention that $\min \emptyset = \infty$). 
Now consider the measure on  $\mathcal{P}$, 
\begin{equation}
\mu(d\pi) = \int_{\mathcal{S}}\varrho_{\bf s}(d\pi)\nu(d{\bf s}),
\label{implicit}
\end{equation}
 where $\varrho_{\bf s}$ is the law of Kingman's paint-box based on ${\bf s}$ (cf. Chapter 2 of Bertoin \cite{bertfragbook}). It is known that $\mu$ is  an exchangeable partition measure meaning that it is invariant under the action of finite permutations on $\mathcal{P}$.
It is also known (cf. Chapter 3 of Bertoin \cite{bertfragbook}) that it is possible to construct a fragmentation process  on the space of partitions $\mathcal{P}$ with the help of  a Poisson point process on $\mathcal{P}\times \mathbb{N}$, $\{(\pi(t), k(t)) : t\geq 0\}$,  which has intensity measure $\mu\otimes\sharp$. The aforementioned   $\mathcal{P}$-valued fragmentation process is a  Markov process which we denote by  $\Pi = \{\Pi(t) : t\geq 0\}$, where 
 $\Pi(t) = (\Pi_1(t), \Pi_2(t),\cdots)$ 
 is such that at all times $t\geq0$ for which an atom $(\pi(t), k(t))$ occurs in $(\mathcal{P}\backslash\mathbb{N})\times \mathbb{N}$, $\Pi(t)$ is obtained from $\Pi(t-)$ by partitioning the $k(t)$-th block  into the sub-blocks $(\Pi_{k(t)}(t-) \cap \pi_j(t) : j=1,2,\cdots)$.
Thanks to the properties of the exchangeable partition measure $\mu$ 
it can be shown that, for each $t\geq 0$, the distribution of $\Pi(t)$ is exchangeable and moreover,  blocks of $\Pi(t)$  have asymptotic frequencies in the sense that for each $i\in\mathbb{N}$, 
\[
|\Pi_i(t)|:=\lim_{n\uparrow\infty} \frac{1}{n}\sharp\{\Pi_i(t)\cap \{1,\cdots, n\} \}
\]
exists almost surely. Further,  the ranked ordering of these asymptotic frequencies form a homogenous mass fragmentation process with dislocation measure $\nu$. 

For future reference we also note from Proposition 2.8 of Bertoin \cite{bertfragbook},  that the ordered asymptotic frequencies of $\pi$ sampled under $\varrho_{\bf s}$, written $|\pi|^\downarrow$, satisfy $|\pi|^\downarrow={\bf s}$ almost surely  and $|\pi_1|$ is a size-biased sample of ${\bf s}$ almost surely. Here `size-biased sample' means that $\varrho_{\bf s}(|\pi_1| = s_i) =s_i$ for $i=1,2,\cdots$. It thus follows from (\ref{implicit}) and Fubini's theorem that for non-negative test functions $f:[0,1]\rightarrow [0,\infty)$ with $f(0)=0$ and $g:\mathcal{S}\rightarrow [0,\infty)$, 
\begin{eqnarray}
\int_{\mathcal{P}} g(|\pi|^\downarrow) f(|\pi_1|)\mu(d\pi)
&=&\int_{\mathcal{S}}  \int_{\mathcal{P}} g({\bf s})f (|\pi_1|)\varrho_{\bf s}(d\pi)\nu(d{\bf s})\nonumber\\
&=&\int_{\mathcal{S}} g({\bf s}) \left( \int_{\mathcal{P}}  f(|\pi_1|)\varrho_{\bf s}(d\pi)\right)\nu(d{\bf s})\nonumber\\
&=&\int_{\mathcal{S}} g({\bf s}) \left(\sum_{i=1}^\infty s_i f(s_i)\right)\nu(d{\bf s}).\label{mutonu}
 \end{eqnarray}


 In the forthcoming discussion, unless otherwise stated,  we shall exclusively understand $(\mathbf{X}, \mathbb{P})$  to be a homogenous mass fragmentation processes as described above and refer to the underlying $\mathcal{P}$-valued fragmentation process as $\Pi$.

Let us introduce the constant 
\[
\underline{ p}: = \inf\left\{  p\in \mathbb{R} : \int_{\mathcal{S}} \left| 1- \sum_{i=1}^\infty s_i^ {p+1} \right| \nu(d{\bf s}) <\infty\right\}
\]
which is necessarily in $(-1,0]$.
 It is well known that
\[
\Phi( p) = \int_{\mathcal{S}} \left(  1- \sum_{i=1}^\infty s_i^ {p+1} \right)\nu(d{\bf s}) 
\]
is strictly increasing and concave for $ p\in(\underline{p},\infty)$. Let us assume the following. 

\medskip

{\bf (A1):} If $\underline{p}=0$ then  
$$\Phi'(0+) =\int_{\mathcal{S}}\left(\sum_{i=1}^\infty s_i \log \left(\frac{1}{s_i}\right)\right)\nu(d{\rm s})<\infty.
$$

\medskip

The function $\Phi$ has a special meaning in the context of the growth of a typically `tagged' fragment.  If one considers the process $\xi =\{\xi_t : t\geq 0\}$, where 
\[
\xi_t := -\log |\Pi_1(t)|,
 \]
  then the underlying Poissonian structure implies that $\xi$ is a subordinator. 
Moreover $\Phi$ turns out to be its Laplace exponent meaning that   $$\Phi(p) := -t^{-1}\log \mathbb{E}(e^{-p\xi_t})$$ for all $p\geq 0$. The subordinator $\xi$ is  killed at rate $\Phi(0)\geq 0$ (with zero killing rate meaning that there is no killing). Further, when $\underline{p}<0$, $\xi$ has finite mean, that is to say $\Phi'(0)<\infty$ and then same is true when $\underline{p}=0$ thanks to (A1).

 Through the exponent $\Phi$ we may introduce the Malthusian Parameter ${ p^*}$ which is the unique solution to the equation $\Phi(p^*)=0$ when it exists. In the conservative case, it always exists and satisfies $p^*=0$. For the dissipative case, we introduce an extra assumption to cater for its existence. 

\medskip

{\bf (A2):} If $\nu$ is dissipative, there exists a ${ p^*}>\underline{ p}$ such that $\Phi(p^*)=0$.

\medskip

Note that necessarily in the dissipative case $p^*<0$ and then $\Phi(0)>0$.
A second assumption we will need with with regard to the Malthusian parameter in the dissipative case is the following.

\medskip

{\bf (A3):} If $\nu$ is dissipative then there exists a $p_0\in(1,2]$ such that 
\[
\int_{\mathcal{S}} \left(\sum_{i=1}^\infty s_i^{ 1+p^*}\right)^{p_0}\nu(d{\bf s})<\infty.
\]

\vspace{.5 cm}

\noindent {\it From this moment on we shall always assume, unless otherwise stated that assumptions (A1), (A2) and (A3) are in force.}

\vspace{.5 cm}

Let $\mathbf{X}_\eta:=\{X_{\eta, j} : j\geq 1\}$ be an arbitrary enumeration of fragments as they are frozen at the instant they become strictly smaller than $1\geq \eta>0$.  The set $\mathbf{X}_\eta$ is a classic example of the resulting family of fragments obtained when stopping the fragmentation process ${\mathbf{X}}$ at a stopping line (corresponding to the first fragment in its line of decent to be smaller than $\eta$ in size) for which the so-called extended fragmentation property holds. The latter says that, if $\mathcal{F}_\eta$ is the natural filtration generated by sweeping through the fragmentation process up to the stopping line  ${\bf X}_\eta$, then given $\mathcal{F}_\eta$ the subsequent evolution of the fragments in $\mathbf{X}_\eta$  are independent and are copies of $\mathbf{X}$ with respective laws $\mathbb{P}_{X_{\eta,1}}, \, \mathbb{P}_{X_{\eta, 2}},\cdots$. See  Definition 3.4 and Theorem 3.14  of Bertoin \cite{bertfragbook} respectively.  Moreover we may talk of $\{\mathbf{X}_\eta : 1\geq\eta>0\}$  as a monotone sequence of stopping lines since, if $\eta'<\eta$, then every fragment in $\mathbf{X}_{\eta'}$ is either a fragment of $\mathbf{X}_{\eta}$ or the result of a sequence of dislocations of a fragment in $\mathbf{X}_\eta$. For convenience, when $\eta\geq 1$ we shall simply define $\mathbf{X}_\eta$ as a block of unit size.

Bertoin and Martinez \cite{energy} propose the idea of using a fragmentation process to model the crushing of rocks in the mining industry. In that setting, one assumes  that rock fragments, occurring as the consequence of subjecting a single boulder to a continuous crushing process, are representation by the evolution of  $\mathbf{X}$. If rock fragments are no longer subject to crushing the moment that they are small enough to pass through a mesh of fixed diameter, then one may think of  the stopping line $\mathbf{X}_\eta$ as the sizes of what falls through the mesh throughout the entire crushing process.

In the setting of a C-M-J process, where birth times correspond to the negative logarithm of fragment sizes, the analogue of the stopping line $\mathbf{X}_\eta$ is what is known as the {\it coming generation}.

\section{Main result}

The main object of interest in this paper is the following family of random measures on $[0,1]$
\begin{equation}
\rho_\eta(\cdot) = \sum_{j} X_{\eta, j}^{1+p^*} \delta_{X_{\eta, j}/\eta}(\cdot), \qquad 1\geq\eta>0,
\label{empirical}
\end{equation}
where $\delta(\cdot)$ is the dirac measure, assigning unit mass to sets which contain the point $x$.
Here we assume that (A2) is in force.
For all bounded measurable functions $f:[0,1]\rightarrow [0,\infty)$ we write
\[
\langle\rho_\eta, f\rangle =  \sum_{j} X_{\eta, j}^{1+p^*} f(X_{\eta, j}/\eta). 
\]
Our objective is to show  $\langle\rho_\eta, f\rangle$ behaves, as $\eta\downarrow 0$, like the limit of a classical, unit mean martingale, up to a multiplicative constant which depends on $f$.

The aforementioned classical unit mean martingale is  $\{\sum_{i=1}^\infty X^{1+p^*}_i(t), t\geq 0\}$; it is the analogue of the classical additive martingale for branching random walks. In Theorem 1 of \cite{bertoinrouault1-arxiv} the analogue of the classical Biggins' Martingale Convergence Theorem was proved which shows in particular that  
\[
\Lambda(p^*):=\lim_{t\uparrow\infty} \sum_{i=1}^\infty X_i^{1+p^*}(t)
\]
exists in the $L^1(\mathbb{P})$ sense (in addition to being a $\mathbb{P}$-almost sure limit). Note, in the conservative case we have trivially that $\Lambda(p^*)=1$ on account of the fact that $p^*=0$ and total mass is preserved at all times, $\sum_{i=1}^\infty X_i(t) =1$ for all $t\geq 0$.
As a prelude to the main result, the next lemma shows us that the total mass of the empirical measure (\ref{empirical}) is a uniformly integrable martingale with the same limit $\Lambda(p^*)$.

\begin{lemma}\label{firstlemma} We have that  
\[
\langle\rho_\eta,1\rangle = \mathbb{E}(\Lambda(p^*)|\mathcal{F}_\eta)\qquad \mbox{ for }1\geq \eta>0,
\]
showing in particular that $\{\langle\rho_\eta,1\rangle : 1\geq\eta>0\}$ is a uniformly integrable, unit mean martingale. 
\end{lemma}
\begin{proof}
Suppose that $\mathcal{A}_{\eta}(t)$ corresponds to the indices of fragments in $\mathbf{X}(t)$ which are neither fragments in $\mathbf{X}_\eta$ nor descendent  from fragments in the stopping line $\mathbf{X}_\eta$; that is to say, $\mathcal{A}_{\eta}(t)$ corresponds to the indices of fragments in $\mathbf{X}(t)$ which are greater than or equal to $\eta$ in size.  Also write $\mathcal{D}_\eta(t)$ for the indices of fragments in $\mathbf{X}_\eta$ which are either in $\mathbf{X}(t)$ or have descendants in $\mathbf{X}(t)$. Then by the 
extended fragmentation property and the fact that $\mathbb{E}\left(\sum_{i=1}^\infty X^{1+p^*}_i(t)\right)=1$ for all $t\geq 0$, we have
\begin{equation}
 \mathbb{E} \left(\left.\sum_{i=1}^\infty X^{1+p^*}_i(t)\right|\mathcal{F}_\eta\right) = \sum_{i\in \mathcal{A}_{\eta}(t) }X^{1+p^*}_i(t) + \sum_{j\in\mathcal{D}_{\eta}(t) }X^{1+p^*}_{\eta,j}.
\label{t-infty}
\end{equation}
It is known that the largest fragment decays at an exponential rate (cf. Bertoin \cite{bertoinasym}) and hence there exists an almost surely finite time $T$ such that $\mathcal{A}_{\eta}(t)=\emptyset$ (and thus $\mathcal{D}_\eta(t)$ contains all the indices of the stopping line $\mathbf{X}_\eta$) for all $t\geq T$. 
In that case, taking limits as $t\uparrow\infty$ in (\ref{t-infty}) we get 
\begin{equation}
\lim_{t\uparrow\infty}\mathbb{E} \left(\left.\sum_{i=1}^\infty X^{1+p^*}_i(t)\right|\mathcal{F}_\eta\right) 
= \langle\rho_\eta, 1\rangle.
\label{aslimit}
\end{equation}
Note also that, since $\Lambda(p^*)$ is an $L^1(\mathbb{P})$ limit and
\[
\mathbb{E}\left(\left|\mathbb{E} \left(\left.\sum_{i=1}^\infty X^{1+p^*}_i(t)\right|\mathcal{F}_\eta\right) -\mathbb{E} \left(\left.\Lambda(p^*)\right|\mathcal{F}_\eta\right)  \right|\right)
\leq \mathbb{E}\left(\left|\sum_{i=1}^\infty X^{1+p^*}_i(t)-\Lambda(p^*)\right|\right),
\]
the random variable $\mathbb{E} \left(\left.\Lambda(p^*)\right|\mathcal{F}_\eta\right) $ is the $L^1(\mathbb{P})$ limit of $\mathbb{E} \left(\left.\sum_{i=1}^\infty X^{1+p^*}_i(t)\right|\mathcal{F}_\eta\right)$ as $t\uparrow\infty$. Referring back to (\ref{aslimit}) we deduce that in fact
\[
 \langle\rho_\eta, 1\rangle = \mathbb{E} \left(\left.\Lambda(p^*)\right|\mathcal{F}_\eta\right)
\]
which implies the statement of the lemma.
\hfill\end{proof}

In the conservative case the martingale $\langle\rho_\eta, 1\rangle$ is of course trivially identically equal to $1$ for all $1\geq \eta>0$. In the dissipative case, 
 although the limiting variable $\Lambda(p^*)$ is the result of $L^1(\mathbb{P})$ convergence, it is not immediately clear  that $\mathbb{P}(\Lambda(p^*) >0)=1$. However  by conditioning on the state of the fragmentation process at time $t>0$, one easily shows that if $\phi(x)=\mathbb{P}_x(\Lambda(p^*) =0)$ for any $1\geq x>0$, then 
\[
\phi(x) = \mathbb{E}_x\left(\prod_{i=1}^\infty \phi(X_i(t))\right).
\]
Note however that by homogeneity, for all $1\geq x>0$, $\mathbb{P}_x(\Lambda(p^*) =0) = \mathbb{P}(x\Lambda(p^*) =0) = \mathbb{P}(\Lambda(p^*) =0)$. It follows that if $\mathbb{P}(\Lambda(p^*) =0)<1$ then $\phi(x)=0$ for all $1\geq x>0 $. In partcular, $\mathbb{P}(\Lambda(p^*) =0)=0$. The only other possibility is that $\mathbb{P}(\Lambda(p^*) =0)=1$ which contradicts the fact that $\Lambda(p^*)$ is an $L^1(\mathbb{P})$-limit.

Now consider the completely deterministic measure $\rho$ on $[0,1]$ which is absolutely continuous with respect to Lebesgue measure and satisfies 
\[
\rho(du) = \frac{1}{\Phi'(p^*)}\left(\int_{\mathcal S}\mathbf{1}_{\{u>s_n\}} s^{1+p^*}_n \nu(d{\bf s})\right) \frac{du}{u} 
\]
where, in the case that $p^*=0$ we understand $\Phi'(p^*)=\Phi'(0+)$.
For bounded measurable functions $f:[0,1]\rightarrow[0,\infty)$ we write 
\begin{equation}
\langle\rho, f\rangle = \frac{1}{\Phi'(p^*)}\int_0^1 f(u)\left(\int_{\mathcal S}\mathbf{1}_{\{u>s_n\}} s^{1+p^*}_n \nu(d{\bf s})\right) \frac{du}{u}
\label{usemct}
\end{equation}

The measure $\rho$ has a special meaning for the subordinator $\xi$. Let us suppose that $\mathbf{P}$ is taken as the intrinsic law associated with $\xi$. On account of the fact that  $\Phi(p^*)=0$, the quantity $\Phi^{(p^*)}(q):=\Phi(q+p^*)$ for $p\geq 0$ represents the Laplace exponent of $\xi$ under the change of measure 
\[
\left.\frac{d\mathbf{P}^{(p^*)}}{d\mathbf{P}}\right|_{\sigma\{\xi_s: s\leq t\}} = e^{- p^* \xi_t}.
\]
If $m$ is the  jump measure associated with  $(\xi, \mathbb{P})$ and $m^{(p^*)}$ is the jump measure associated with $(\xi, \mathbb{P}^{(p^*)})$, then standard theory tells us that $m^{(p^*)}(dx) = e^{-p^* x}m(dx)$ for $x>0$.  In terms of the dislocation measure $\nu$ we may thus write (cf. p. 142 of Bertoin \cite{bertfragbook}) that 
\begin{equation}
m^{(p^*)}(dx) = e^{-x(1+p^*)}\sum_{i=1}^\infty\nu(-\log s_i \in dx)
\label{Lambda}
\end{equation}
for $x\in(0,\infty)$. It is also known from classical renewal theory (cf. Bertoin \cite{subordinators}) that as $\mathbf{E}^{(p^*)}(\xi_1) = \Phi'(p^*)<\infty$,  
\[
\lim_{x\uparrow\infty} \mathbf{E}^{(p^*)} (\xi_{\tau(x)}-x \in dz) =\frac{1}{\Phi'(p^*)} m^{(p^*)}((z,\infty)) dz
\]
where $\tau(x)  =\inf\{t>0: \xi_t >x\}$. It is therefore straightforward to show, using (\ref{Lambda}) and a change of variables,  that for any bounded measurable $f:[0,1]\rightarrow[0,\infty)$,
\begin{equation}
\lim_{x\uparrow\infty} \mathbf{E}^{(p^*)} f(e^{-(\xi_{\tau(x)}-x)}) =\langle\rho,f\rangle.
\label{sub-remark}
\end{equation}

Our main theorem, below, relates the limiting behaviour of $\rho_\eta$ to $\rho$.

\begin{theorem}[Strong Law of Large Numbers]\label{main}
For any homogenous fragmentation process 
we have
\[
\lim_{\eta_{\downarrow 0}}  \frac{\langle\rho_\eta, f\rangle  }{\langle\rho, f\rangle }
= \Lambda(p^*) \qquad \mathbb{P}\mbox{-a.s.}
\]
for all   bounded measurable functions $f:[0,1]\rightarrow [0,\infty)$.
\end{theorem}


The above result can also be rephrased in a slightly different way which is more in line with classical results of this type for spatial branching processes. As a consequence of the forthcoming Lemma \ref{many-to-one} it can easily be shown that 
\[
\mathbb{E}(\langle\rho_\eta,f\rangle) = \mathbf{E}^{(p^*)}[f(e^{-(\xi_{\tau(-\log \eta)}  +\log\eta  )})],
\]
in which case, taking note of (\ref{sub-remark}), the statement of the theorem may also be read as
\[
\lim_{\eta\downarrow0}\frac{\langle\rho_\eta,f\rangle}{\mathbb{E}(\langle\rho_\eta,f\rangle) }=\Lambda(p^*) 
\]
$\mathbb{P}$-almost surely.

Theorem \ref{main} extends Corollary 2 of Bertoin and Martinez \cite{energy}  
where $L^2(\mathbb{P})$ convergence had been established in the conservative case. See also Proposition 1.12 of \cite{bertfragbook}. It  also lends itself to classical strong laws of large numbers that were proved by Nerman \cite{nerman} in the setting of C-M-J processes, Asmussen and Herring \cite{AH76, AH77}  in the context of spatial branching processes  and, more recently, Engl\"ander, Harris and Kyprianou \cite{EHK}, Chen and Shiozawa \cite{CS} and Chen, Ren and Wang \cite{CRW} for branching diffusions and Engl\"ander \cite{E} and Engl\"ander and Winter \cite{EW} for superdiffusions.

Theorem \ref{main} also gives rise to a strong law of large numbers for self-similar fragmentation processes. 
Recall that any self-similar fragmentation process with a given dislocation measure  $\nu$ and index $\alpha\in\mathbb{R}$ may be obtained from the associated homogenous process via the use of stopping lines. Indeed it is known (cf. Bertoin \cite{bertoinselfsimilar}) that, in the $\mathcal{P}$-valued representation of a given homogenous fragmentation process,  if one defines the stopping times 
\[
\theta_i(t) =\inf\left\{ u\geq 0 : \int_0^u |\Pi_i(s)|^{-\alpha} ds > t\right\}
\]
for each $i\in\mathbb{N}$, then stopping the block in $\Pi$ which contains the integer $i$ at $\theta_i(t)$ for each $i\in\mathbb{N}$ produces a stopping line of fragments whose asymptotic frequencies exist. We    denote the associated mass fragments of the latter by $\mathbf{X}^{(\alpha)}(t)$. As a process in time, $\mathbf{X}^{(\alpha)}:= \{\mathbf{X}^{(\alpha)}(t) : t\geq 0 \}$ is defined to be the $\alpha$-self-similar fragmentation process associated to the dislocation measure $\nu$. 
Now returning to the definition of $\mathbf{X}_\eta$, one notes that, as a stopping line concerned with fragment sizes, it is blind to any time-changes made along individual nested sequences of fragments. Therefore, if one considers in the process $\mathbf{X}^{(\alpha)}(t)$ the  first fragments in their line of decent to be smaller than $\eta$ in size, then one obtains precisely $\mathbf{X}_\eta$ again. We thus obtain the following corollary to Theorem \ref{main}.


\begin{corollary} Theorem \ref{main} is valid for any self-similar fragmentation process.
\end{corollary}

We conclude this section by noting that the remainder of the paper is set out into two further sections. Some initial preliminary results followed by the proof of the main result.

\section{Preliminary results}

In this section we produce a number of initial results which will be collectively used in the proof of Theorem \ref{main}. 
A key element of much of the reasoning is the now very popular {\it  spine decomposition} which we briefly recall here; for a fuller account the reader should consult Bertoin and Rouault \cite{bertoinrouault1, bertoinrouault1-arxiv} or Bertoin \cite{bertfragbook} however. 

The process
 \[
\Lambda_t( p):= \sum_{i=1}^\infty X_i^{1+p}(t)e^{\Phi(p)t}, \qquad t\geq 0
 \]
 is a martingale for all $p\in(\underline{p}, \infty)$  which (thanks to positivity) converges almost surely to its limit which we denote by $\Lambda(p)$.  When $p=p^*$ it is the same martingale discussed in the previous section.  For each $p\in(\underline{p},\infty)$ define the measure $\mathbb{P}^{(p)}$ via
\begin{equation}
\left.\frac{d\mathbb{P}^{(p)}}{d\mathbb{P}}\right|_{\sigma\{\mathbf{X}(s): s\leq t\}} = \sum_{i=1}^\infty X^{1+p}_i(t) e^{\Phi(p)t}.
\label{COM}
\end{equation}

Following the classical analysis of Lyons \cite{lyons}, Bertoin and Rouault \cite{bertoinrouault1, bertoinrouault1-arxiv} show that the process $(\mathbf{X}, \mathbb{P}^{(p)})$ is equal in law to the ranked asymptotic frequencies of an $\mathcal{P}$-valued fragmentation process with a distinguished nested sequence of fragments known as a {\it spine}. The evolution of the latter is the same as the process $\Pi$ except for a modification to the way in which the block containing the integer 1, $\Pi_1(t)$, evolves for $t\geq 0$. This is done by working with a Poisson point process on $\mathcal{P}\times\mathbb{N}$ with intensity adjusted to be equal to 
\[
(\mu^{(p)}\otimes\sharp)|_{\mathcal{P}\times \{1\}}+(\mu\otimes\sharp)|_{\mathcal{P}\times \{2,3,\cdots\}} \mbox{ where  }\mu^{(p)}(d\pi) = |\pi_1|^{p}\mu(d\pi).
\]
An important consequence of this change of measure is that  $\{ -\log |\Pi_1(t)|: t \geq 0\}$, has the characteristics of an exponentially tilted subordinator.
Indeed note that with the help of a telescopic sum, an application of  the compensation formula for Poisson point processes and (\ref{mutonu}), we have for $\lambda, t\geq 0$
\begin{eqnarray*}
\mathbb{E}^{(p)} (|\Pi_1(t)|^{\lambda}) &=& \mathbb{E}^{(p)}\left(\sum_{s\leq t}(1- |\pi_1(s)|^\lambda) |\Pi_1(s-)|^{\lambda}\right)\\
&=&\int_0^t \mathbb{E}^{(p)}(|\Pi_1(s)|^{\lambda})ds \int_{\mathcal{P}} (1- |\pi_1|^\lambda)\mu^{(p)}(d\pi)\\
&=&\int_0^t \mathbb{E}^{(p)}(|\Pi_1(s)|^{\lambda})ds \int_{\mathcal{S}} \sum_{i=1}^\infty s^{1+p}_i(1- s_i^\lambda)\nu(d{\bf s})\\
&=&\Phi^{(p)}(\lambda)\int_0^t \mathbb{E}^{(p)}(|\Pi_1(s)|^{\lambda})ds
\end{eqnarray*}
where $\Phi^{(p)}(\lambda) = \Phi(\lambda + p) - \Phi(p)$.
Solving the above integral equation for $\mathbb{E}^{(p)} (|\Pi_1(t)|^{\lambda})$ shows that the intrinsic law of $-\log |\Pi_1(\cdot)|$
 is precisely that of  $(\xi,\mathbf{P}^{(p)})$ where
\[
\left.\frac{d\mathbf{P}^{(p)}}{d\mathbf{P}}\right|_{\sigma\{\xi_s : s\leq t\}} = e^{- p\xi_t +\Phi(p)t}.
\]

We also mention that the natural analogue of Biggins' martingale convergence theorem proved in Theorem 1 of \cite{bertoinrouault1-arxiv} states that there is $L^1(\mathbb{P})$ convergence of the martingale to $\Lambda(p)$ if and only if $p\in(\underline{p}, \overline{p})$ where $\overline{p}$ is the unique solution to the equation $(1+p)\Phi'(p) = \Phi(p)$ and otherwise $\Lambda(p)=0$ almost surely. Note in particular that, since $p\geq \overline{p}$ if and only $(1+p)\Phi'(p)\leq \Phi(p)$, it follows that $p^*< \overline{p}$.

Our first preparatory lemma  is sometimes referred to as a `many-to-one' identity (see for example \cite{harris1, harris2}) and has appeared in many guises throughout the study of spatial branching processes as one sees, for example, in early work such as  Bingham and Doney \cite{BD75} on Crump-Mode-Jagers processes.

\begin{lemma}\label{many-to-one}
For all measurable $f:[0,1]\rightarrow [0,\infty)$ we have 
\[
\mathbb{E}\left(\sum_j X^{1+p^*}_{\eta, j} f(X_{\eta, j})\right) = \mathbf{E}^{(p^*)}(f(e^{-\xi_{\tau(-\log \eta)}})).
\]
\end{lemma}
\begin{proof}
It will be convenient to introduce a third representation of  fragmentation processes. Denote by $\Theta = \{\Theta(t) : t\geq 0\}$ the interval representation of the fragmentation process (see \cite{bad}). That is to say $\Theta$ is a sequence of nested random open subesets of $(0,1)$ in the sense that $\Theta(t+s)$ is a refinement of $\Theta(t)$ for $s,t\geq 0$. The process ${\bf X}(t)$ can be recovered from $\Theta(t)$ by ranking the  sequence of the lengths of the intervals of which $\Theta(t)$ is comprised.

Next define for each $u\in(0,1)$
\[
T_\eta(u)=\inf\{t\geq 0 : |I_u(t)|<\eta\},
\]
where $|I_u(t)|$ denotes the length of the component of $\Theta(t)$ containing $u$. We may now write
\begin{eqnarray*}
\mathbb{E}\left(\sum_j X^{1+p^*}_{\eta, j} f(X_{\eta, j})\right) &=& \mathbb{E}\int_0^1 
|I_u(T_\eta(u))|^{p^*}
f(I_u(T_\eta(u)))du
\\
& = &\mathbb{E}\left[ 
|I_U(T_\eta(U))|^{p^*}
f(I_U(T_\eta(U)))\right],
\end{eqnarray*}
where $U$ is an independent and uniformly distributed random variable. It is known however (cf. Bertoin and Martinez \cite{energy}) that the process $\{-\log |I_U(t)|:t\geq 0\}$ under $\mathbb{P}$  has the same law as the subordinator $\xi$ under $\mathbf{P}$ and hence we have 
\begin{eqnarray*}
\mathbb{E}\left(\sum_j X^{1+p^*}_{\eta, j} f(X_{\eta, j})\right) &=&
\mathbf{E}\left[ e^{-p^*\xi_{\tau(-\log \eta)}}f(e^{-\xi_{\tau(-\log \eta)}})\right]\\
&=&
\mathbf{E}^{(p^*)}\left[ f(e^{-\xi_{\tau(-\log \eta)}})\right],
\end{eqnarray*}
thus completing the proof.
\hfill
\end{proof}

Next we  define for $1\geq\eta>0$ the  families of subindices of the stopping line $\mathbf{X}_\eta$
\[
\mathcal{J}_{\eta,s} =\{j=1,2,\cdots : X_{\eta, j} \geq \eta^s\} \mbox{ and }\mathcal{J}^{\rm c}_{\eta,s} =\{j=1,2,\cdots : X_{\eta, j} <\eta^s\}
\]
Note in particular that $\mathcal{J}_{\eta, 1}^{\rm c}$ is the full set of indices of fragments in $\mathbf{X}_\eta$.

\begin{lemma}\label{little-bits} There exists an $s_0>1$ such that for all $s\geq s_0$
\[
\lim_{\eta\downarrow 0}\sum_{j\in\mathcal{J}_{\eta,s}} X^{1+p^*}_{\eta, j} = \Lambda(p^*) 
\]
$\mathbb{P}$-almost surely.
\end{lemma}
\begin{proof}
We deal with the conservative and dissipative cases separately.  Consider first the conservative case, $p^*=0$. In the terminology of the proof of the previous Lemma we may easily write
\begin{equation}
\sum_{j\in\mathcal{J}^{\rm c}_{\eta,s}} X_{\eta, j} = \int_0^1  \mathbf{1}_{\{ |I_u(T_\eta(u))| <\eta^s\}}du.
\label{useDCT}
\end{equation}

Let us temporarily write  $ \xi_u(t)=-\log |I_u(t)|  $, $x=-\log \eta$ and let $\tau_u(x)  =\inf\{t>0: \xi_u(t)>x\} $.  Note that the event $\{ |I_u(T_\eta(u))| < \eta^s\}$ is equivalent to the event 
\[
\left\{
\frac{\xi_u(\tau_u(x)) -x}{x} >(s-1)
 \right\}.
 \]
 Since $\xi_u$ is a subordinator with Laplace exponent $\Phi$, and thus has finite mean by assumption (A1), it follows from the classical theory of subordinators (cf. Bertoin \cite{subordinators}) that 
 \[
 \lim_{x\uparrow\infty} \frac{\xi_u(\tau_u(x)) -x}{x} =0 
 \]
 almost surely.
 The result now follows by an application of the Dominated Convergence Theorem in (\ref{useDCT}).
 
 Next consider the dissipative case so that $p^*<0$. In that case with $p\in(\underline{p}, p^*)$ we may develop (\ref{useDCT}) as follows,
 \begin{eqnarray}
 \sum_{
 j\in\mathcal{J}^{\rm c}_{\eta,s}
 } X^{1+p^*}_{\eta, j} &\leq&
\eta^{s(p^*-p)}   \sum_{j\in\mathcal{J}^{\rm c}_{\eta,s}} X^{1+p}_{\eta, j} e^{\Phi(p)\sigma_{\eta,j} -\Phi(p)\sigma_{\eta,j}}\nonumber\\
&\leq &\eta^{s(p^*-p)} e^{-\Phi(p)\sigma_\eta}  \sum_{j} X^{1+p}_{\eta, j} e^{\Phi(p)\sigma_{\eta,j} },
\label{firstestimate}
 \end{eqnarray}
 where $\sigma_{\eta, j}$ is the time that the block $X_{\eta, j}$ enters the stopping time $\mathbf{X}_\eta$ and $ \sigma_\eta = \inf\{ t>0 :  X_1(t)<\eta\}$, the first time the largest block becomes smaller than $\eta$. We claim that  $\sigma_\eta<\infty$  almost surely finite. 
Indeed the almost sure finiteness of $\sigma_\eta$ is a consequence of the classical result that the speed of the largest particle obeys the strong law of large numbers 
\[
\lim_{t\uparrow\infty}-\frac{\log X_1(t)}{t} = \Phi'(\overline{p})>0 
\]
$\mathbb{P}$-almost surely. See for example Corollary 1.4 of Bertoin \cite{bertfragbook}. (The latter result was proved for homogenous fragmentation processes with finite dislocation measure, however the proof passes through verbatim for the case of infinite dislocation measure).
Note also that since $X(\cdot)$ is a right continuous process with inverse $\sigma_{(\cdot)}$, standard arguments for right continuous monotone functions and their inverses imply that 
\[
\lim_{\eta\downarrow 0}-\frac{\sigma_\eta}{    \log \eta} =\Phi'(\overline{p})
\]
$\mathbb{P}$-almost surely. We therefore may proceed with the estimate (\ref{firstestimate}) and deduce that for all $\eta$ sufficiently small and $\epsilon>0$
\[
 \sum_{
 j\in\mathcal{J}^{\rm c}_{\eta,s}
 } X^{1+p^*}_{\eta, j} 
 \leq 
 \eta^{s(p^*-p) + (1+\epsilon )\Phi(p)/\Phi'(\overline{p})} \sum_{j} X^{1+p}_{\eta, j} e^{\Phi(p)\sigma_{\eta,j} }.
\]
Using similar arguments to those found in the proof of Lemma \ref{firstlemma} one may show that the random sum on the right hand side above is the projection of the $L^1(\mathbb{P})$-martingale limit $\Lambda(p)$ onto the filtration $\mathcal{F}_\eta$ and therefore is a uniformly integrable martingale for $1\geq \eta >0$. It now follows that by choosing $s>-(1+\epsilon) \Phi(p)[(p^*-p)\Phi'(\overline{p})]^{-1}$ (recall that $\Phi(p)<0$) we have
\[
\lim_{\eta\downarrow 0} \sum_{
 j\in\mathcal{J}^{\rm c}_{\eta,s}
 } X^{1+p^*}_{\eta, j} =0
\]
as required. \hfill\end{proof}

The previous lemma allows us to establish the following result.

\begin{lemma}\label{a.s.approx} For all bounded measurable $f:[0,1]\rightarrow [0,\infty)$ and all  $s$ sufficiently large,
\[
\lim_{\eta\downarrow 0} |\mathbb{E} (\langle \rho_{\eta^s}, f\rangle |\mathcal{F}_{\eta}) - \langle\rho, f\rangle\Lambda(p^*)|
\]
$\mathbb{P}$-almost surely.
\end{lemma}
\begin{proof}
We introduce the notation
\[
\psi(x,\eta) = \mathbf{E}^{(p^*)}(f(e^{-\xi_{\tau(-\log x)}}/\eta))
\]
for bounded measurable $f:[0,1]\rightarrow [0,\infty)$
noting in particular by (\ref{sub-remark}) that $\lim_{\eta\downarrow 0}\psi(\eta,\eta)=\langle\rho, f\rangle$.
For convenience we  prove the result when $\eta$ is replaced by $\eta^2$ or equivalently with $s$ replaced by $2s$.
By splitting the fragments in $\mathbf{X}_{\eta^{2s}}$ into descendants of fragments in $\mathbf{X}_{\eta}$ whose mass is no smaller than $\eta^s$,  descendants of fragments in $\mathbf{X}_{\eta}$ whose mass lies between $[\eta^{2s}, \eta^s)$ and fragments which belong to both $\mathbf{X}_\eta$ and $\mathbf{X}_{\eta^{2s}}$ we get with the help of the fragmentation property applied at the stopping line $\mathbf{X}_\eta$ and Lemma \ref{many-to-one} that
\begin{eqnarray}
\mathbb{E}(\langle\rho_{\eta^{2s}}, f\rangle|\mathcal{F}_\eta)&=&
\sum_{j\in \mathcal{J}_{\eta, s}}X^{1+p^*}_{\eta, j} \psi(\eta^{2s}/X_{\eta,j}, \eta^{2s}/X_{\eta,j}) \nonumber \\
&& +
\sum_{j\in \mathcal{J}_{\eta, 2s}\backslash\mathcal{J}_{\eta, s}}X^{1+p^*}_{\eta, j}\psi(\eta^{2s}/X_{\eta,j}, \eta^{2s}/X_{\eta,j})
\nonumber \\
&&+\sum_{j\in\mathcal{J}^{\rm c}_{\eta, 2s}}X^{1+p^*}_{\eta, j}f(X_{\eta,j}/\eta^{2s}).
\label{III}
\end{eqnarray}
Note that 
\[
\sup_{j\in \mathcal{J}_{\eta, s}}\eta^{2s}/X_{\eta,j}\leq \eta^s \rightarrow 0\qquad \mathbb{P}\mbox{-a.s.  as $\eta\downarrow 0$}
\]
and hence, using Lemma \ref{little-bits}, the first term on the right hand side of (\ref{III}) converges almost surely to $\langle\rho,f\rangle\Lambda(p^*)$. Without loss of generality we may assume that $f$ is uniformly bounded by $1$ in which case $\psi$ is uniformly bounded by $1$  and then  the second term on the right hand side of (\ref{III}) satisfies
\[
\lim_{\eta\downarrow 0}\sum_{j\in \mathcal{J}_{\eta, 2s}\backslash\mathcal{J}_{\eta, s}}X^{1+p^*}_{\eta, j}\psi(\eta^{2s}/X_{\eta,j}, \eta^{2s}/X_{\eta,j})\leq \lim_{\eta\downarrow 0}\sum_{j\in \mathcal{J}^{\rm c}_{\eta, s}}X^{1+p^*}_{\eta, j}=0
\]
$\mathbb{P}$-almost surely,
where the final equality follows by Lemma \ref{little-bits}. Similar reasoning shows that the third term on the right hand side of (\ref{III}) converges almost surely to zero. 
\hfill\end{proof}

The previous lemma showed that $\mathbb{E} (\langle \rho_{\eta^s}, f\rangle |\mathcal{F}_{\eta})$ is a good approximation for  $\langle\rho, f\rangle$. For a random variable $Y$ we denote the  $L^p(\mathbb{P})$ norm in the usual way,
$
||Y||_p= \mathbb{E}(|Y|^p)^{1/p}.
$
The next Proposition shows that $\mathbb{E} (\langle \rho_{\eta^s}, f\rangle |\mathcal{F}_{\eta})$ is also a good approximation for $\langle\rho_{\eta^{s}}, f\rangle$ but now with respect to the $L^p(\mathbb{P})$-norm. 
\begin{proposition}\label{Lp}
Suppose that $f:[0,1]\rightarrow[0,\infty)$ is bounded and measurable. Then there exist constants  $p\in(1,2]$ and $\kappa_p\in(0,\infty)$ such that 
 for all $s>1$  and $1\geq\eta>0$, 
\[
||\langle\rho_{\eta^{s}}, f\rangle- \mathbb{E} (\langle \rho_{\eta^s}, f\rangle |\mathcal{F}_{\eta})||_p\leq  \kappa_p \eta^{(p-1)(1+p^*)/p}.
\]
\end{proposition}
\begin{proof}
Begin by noting that fragments in $\mathbf{X}_{\eta^s}$ are either descendants of fragments in $\mathbf{X}_\eta$ or belong themselves to $\mathbf{X}_\eta$. With the help of the extended fragmentation property applied at the stopping line $\mathbf{X}_\eta$, this incurs the decomposition
\begin{eqnarray}
\lefteqn{
\langle\rho_{\eta^{s}}, f\rangle- \mathbb{E} (\langle \rho_{\eta^s}, f\rangle |\mathcal{F}_{\eta})
}\nonumber\\
& =&
\sum_{j\in\mathcal{J}^{\rm c}_{\eta, s}}X^{1+p^*}_{\eta,j}f(X_{\eta,j}/\eta) - 
 \mathbb{E}\left(\left.\sum_{j\in\mathcal{J}^{\rm c}_{\eta, s}}X^{1+p^*}_{\eta,j}f(X_{\eta,j}/\eta)\right|\mathcal{F}_\eta\right)\nonumber\\
&&+\sum_{j\in\mathcal{J}_{\eta, s}}X^{1+p^*}_{\eta,j}\left(\Delta_j - \mathbb{E}(\Delta_j|\mathcal{F}_\eta)\right)\nonumber\\
&=&\sum_{j\in\mathcal{J}_{\eta, s}}X^{1+p^*}_{\eta,j}\left(\Delta_j - \mathbb{E}(\Delta_j|\mathcal{F}_\eta)\right),
\label{deltas}
\end{eqnarray}
where, given $\mathcal{F}_\eta$, $\Delta_j$ are independent, each having the same law as $\langle\rho_{\eta_j}, f\rangle$ under $\mathbb{P}$ with $\eta_j=\eta^s/X_{\eta,j}$.

For the next part of the proof we need to make use of two inequalities. The first  is lifted from Lemma 1 of Biggins \cite{big92}. For independent, zero mean random variables $\{Y_1,\cdots, Y_n\}$ with reference expectation operator $E$ and $p\in[1,2]$ we have
\begin{equation}
E\left(\left|\sum_{i=1}^n Y_i\right|^p\right)\leq 2^p \sum_{i=1}^n E(|Y_i|^p).
\label{first}
\end{equation}
Note that an easy calculation using Fatou's Lemma also implies that if one has  an infinite sequence of such variables then 
the same inequality holds except with infinite sums. The second inequality is a direct consequence of Jensen's inequality and says that for all $u,v\in\mathbb{R}$ and $p\geq 1$, 
\begin{equation}
|u+v|^p\leq 2^{p-1}(|u|^p+|v|^p).
\label{second}
\end{equation}
We may now proceed to compute for any $p\in(1,2]$
\begin{eqnarray}
\lefteqn{
\mathbb{E}\left[\left.\left|\langle\rho_{\eta^{s}}, f\rangle- \mathbb{E} (\langle \rho_{\eta^s}, f\rangle |\mathcal{F}_{\eta})\right|^p\right|
\mathcal{F}_{\eta}
\right]
}\nonumber\\
&\leq & 2^p\sum_{j\in\mathcal{J}_{\eta, s}}X_{\eta,j}^{(1+p^*)p}
\mathbb{E}
\left(\left.
\left|\Delta_j - \mathbb{E}(\Delta_j|\mathcal{F}_\eta)
\right|^p
\right| \mathcal{F}_\eta \right)\nonumber\\
&\leq& 2^{2p-1} 
\sum_{j\in\mathcal{J}_{\eta, s}}X_{\eta,j}^{(1+p^*)p}
\mathbb{E}
\left(\left.
\Delta_j ^p+ \mathbb{E}(\Delta_j|\mathcal{F}_\eta)
^p
\right| \mathcal{F}_\eta \right),
\label{expectations-again}
\end{eqnarray}
where in the first inequality we have used (\ref{deltas}) and  (\ref{first}) for infinite sums and in the second inequality we have used (\ref{second}). We spilt the remainder of the proof into the cases that $\nu$ is conservative and dissipative, respectively.

Suppose first that $\nu$ is conservative. Without loss of generality we may again assume that $f$ is uniformly bounded by 1, in which case $\Delta_j\leq \langle\rho_{\eta_j}, 1\rangle= 1$ for all $j\in\mathcal{J}_{\eta, s}$. It follows from (\ref{expectations-again}) that 
\begin{eqnarray}
\mathbb{E}\left[\left|\langle\rho_{\eta^{s}}, f\rangle- \mathbb{E} (\langle \rho_{\eta^s}, f\rangle |\mathcal{F}_{\eta})\right|^p\right]
&\leq& 2^{2p} \mathbb{E}\left(\sum_{j\in\mathcal{J}_{\eta, s}}X_{\eta,j}^{(1+p^*)p}\right)\nonumber\\
&\leq&  2^{2p}\eta^{(p-1)(1+p^*)}\mathbb{E}\left(\sum_j X^{1+p^*}_{\eta,j}\right)\nonumber\\
&=&  2^{2p}\eta^{(p-1)(1+p^*)},
\label{similarto}
\end{eqnarray}
where we have used that $\sum_j X^{1+p^*}_{\eta,j} =\sum_j X_{\eta,j} =1.$
Taking the $p$-th  root  completes the proof for the conservative case.

Next assume that $\nu$ is dissipative. Continuing from the last inequality in (\ref{expectations-again}) we may apply Jensen's inequality to deduce that 
\begin{equation}
\mathbb{E}\left[\left|\langle\rho_{\eta^{s}}, f\rangle- \mathbb{E} (\langle \rho_{\eta^s}, f\rangle |\mathcal{F}_{\eta})\right|^p
\right]
\leq 
2^{2p} 
\mathbb{E}\left[\sum_{j\in\mathcal{J}_{\eta, s}}X_{\eta,j}^{(1+p^*)p}
\mathbb{E}
(\Delta_j ^p| \mathcal{F}_\eta )
\right].
\label{comebackto}
\end{equation}
In order to proceed with the estimate on the right hand side above and in particular to deal with the terms $\mathbb{E}
(\Delta_j ^p| \mathcal{F}_\eta )$, it is first necessary to make an estimate on the quantity $\sup_{1\geq \eta>0}\mathbb{E}(\langle \rho_\eta, f\rangle^p)$.  As usual we assume without loss of generality that $f$ is uniformly bounded by 1. We shall also henceforth proceed with our calculations taking $p = p_0$, where $p_0$ was specified in (A3).
Write $q=p-1$.  To complete the proof, we pursue a series of computations inspired by ideas which are found in Hardy and Harris \cite{spine}. We have 
\begin{eqnarray}
\mathbb{E}(\langle \rho_\eta, f\rangle^p)&\leq& \mathbb{E}(\langle \rho_\eta, 1\rangle \langle \rho_\eta, f\rangle^{q})\nonumber\\
&=&\mathbb{E}^{(p^*)}\left[ \mathbb{E}^{(p^*)}(\langle \rho_\eta, f\rangle^{q} | \mathcal{G})\right]\nonumber\\
&\leq &\mathbb{E}^{(p^*)}\left[ \mathbb{E}^{(p^*)}(\langle \rho_\eta, f\rangle | \mathcal{G})^{q}\right]
\label{referbackto}
\end{eqnarray}
where $\mathcal{G} =\sigma\{(\pi(t), k(t)): t\geq 0, \, k(t)=1\} $ is the $\sigma$-algebra generated by the spine  and the final line above follows from Jensen's inequality. Appealing to the spine decomposition described at the beginning of this section, 
and writing $T_\eta=\inf\{t>0 : |\Pi_1(t)|<\eta\}$,
we may further decompose 
\begin{eqnarray*}
\langle \rho_\eta, f\rangle 
&=& |\Pi_1(T_\eta)|^{1+p^*}f(|\Pi_1(T_\eta)|/\eta)\\
&&+\sum_{t\leq T_\eta}\mathbf{1}_{\{k(t)=1\}} |\Pi_1(t-)|^{1+p^*}\sum_{i=2}^\infty |\pi_i(t)|^{1+p^*}\Delta_{i},
\end{eqnarray*}
where, given $\mathcal{G}$, for each $i=2,3,\cdots$ the random variables $\Delta_i$ are independent and have the same distribution as $\langle\rho_{\eta_i}, f\rangle$ under $\mathbb{P}$ with $\eta_i = \eta/( |\Pi_1(t-)||\pi_i(t)|)$. Note in particular that thanks to the uniform boundedness of $f$,  $\Delta_i\leq 1$ for all $i=2,3,\cdots$. Taking account of the inequality $(\sum_{j=1}^\infty a_j)^q \leq \sum_{j=1}^\infty a_j^q$, which holds for any non-negative sequence of numbers $\{a_j: j\geq 1\}$, using again the uniform boundedness of $f$ by unity, referring back to (\ref{referbackto}), making use of the compensation formula for Poisson random measures and recalling that under $\mathbb{P}^{(p^*)}$ the process $-\log |\Pi_1(\cdot)|$ is equal in law to $(\xi, \mathbf{P}^{(p^*)})$, we get 
the following estimate,
\begin{eqnarray}
\lefteqn{\mathbb{E}(\langle \rho_\eta, f\rangle^p)}&&\nonumber\\
&\leq& \mathbb{E}^{(p^*)}\left[ |\Pi_1(T_\eta)|^{(1+p^*)q}+\sum_{t<\infty}\mathbf{1}_{\{k(t)=1\}} |\Pi_1(t-)|^{1+p^*}\left(\sum_{i=1}^\infty |\pi_i(t)|^{1+p^*}\right)^q
\right]\nonumber\\
&\leq&1 +
\int_0^\infty   \mathbf{E}^{(p^*)}(e^{-q(1+p^*)\xi_{t}})dt\cdot
\int_{\mathcal{P}} \left(\sum_{i=1}^\infty |\pi_i|^{1+p^*}\right)^q\mu^{(p^*)}(d\pi)\nonumber\\
&=&1 +
\int_0^\infty  e^{-\Phi^{(p^*)}(q(1+p^*)) t} dt\cdot
\int_{\mathcal{P}} \left(\sum_{i=1}^\infty |\pi_i|^{1+p^*}\right)^q |\pi_1|^{p^*}\mu(d\pi).
\label{strangeintegral}
\end{eqnarray}
Both integrals on the right hand side of (\ref{strangeintegral}) converge. The first integral converges because 
\[
\Phi^{(p^*)}(q(1+p^*))=\Phi(p^* + (p-1)(1+p^*))=\Phi(p(1+p^*)-1)>\Phi(p^*)=0.
\]
The second integral can be shown, with the help of (\ref{mutonu}), to satisfy
 \begin{eqnarray*}
\int_{\mathcal{P}} \left(\sum_{i=1}^\infty |\pi_i|^{1+p^*}\right)^q |\pi_1|^{p^*}\mu(d\pi)
&=&\int_{\mathcal{S}}  \left(\sum_{i=1}^\infty |s_i|^{1+p^*}\right)^q \left(\sum_{i=1}^\infty s_i \cdot s_i^{p^*}\right)\nu(d{\bf s})\\
&=&\int_{\mathcal{S}}  \left(\sum_{i=1}^\infty |s_i|^{1+p^*}\right)^p \nu(d{\bf s}),
 \end{eqnarray*}
which is finite thanks to assumption (A3).

 We have thus shown that there is a constant $K_p$ (which depends only on $p$) such that 
\[
\sup_{1\geq \eta>0}\mathbb{E}(\langle \rho_\eta, f\rangle^p)\leq K_p.
\]

Now returning to (\ref{comebackto}) we may upper estimate $\sup_{j\in\mathcal{J}_{\eta,s}}\mathbb{E}
(\Delta_j ^p| \mathcal{F}_\eta )$ by $K_p$. The remainder of the proof for the dissipative case follows the reasoning that was presented in conjuction with  (\ref{similarto}) for the conservative case, except now we use that $\mathbb{E}(\sum_j X^{1+p^*}_{\eta,j})=1$ instead of just $\sum_j X^{1+p^*}_{\eta,j}=1$.
\hfill\end{proof}

Before proceeding to the proof of the main theorem, let us conclude this section by stating a corollary  of the previous Proposition which, apart from being contemporary with a similar classical result for various spatial branching processes (see for example Hardy and Harris \cite{spine}), is otherwise of no consequence as far as the remainder of the paper is concerned. 

\begin{corollary}
Suppose that $\nu$ is a dissipative measure. Then the martingale $\{\langle\rho_\eta, 1\rangle: 1\geq \eta>0\} $, and hence the martingale $\{\Lambda_t(p^*): t\geq 0\}$, converges in $L^{p}(\mathbb{P})$, where $p= p_0$ and $p_0$ was specified in assumption (A3). 
\end{corollary}
\begin{proof}
The proof is complete, thanks to standard arguments using dominated convergence and Doob's maximal inequality, as soon as it can be shown that 
\begin{equation}
\mathbb{E}\left(\lim_{ \eta \downarrow 0} \langle\rho_{\eta}, 1\rangle^p \right)<\infty.
\label{thenwearedone}
\end{equation}
However, taking $f=1$ in the statement of Proposition \ref{Lp} and noting that $\mathbb{E}(\langle\rho_{\eta_2},1\rangle|\mathcal{F}_{\eta_1}) = \langle\rho_{\eta_1},1\rangle$, we see that when $p=p_0$, for all $1\geq \eta_1\geq\eta_2>0$,
\[
\lim_{\eta_1,\eta_2\downarrow 0}|| \langle\rho_{\eta_1}, 1\rangle- \langle\rho_{\eta_2}, 1\rangle||_p = 0,
\]
showing that $\{\langle\rho_\eta, 1\rangle: 1\geq \eta>0\}$ is a Cauchy family in the space of $L^{p}(\mathbb{P})$ right continuous martingales adapted to $\{\mathcal{F}_\eta: 1\geq \eta>0\}$. It follows that the almost sure limit point $\Lambda(p^*)=\lim_{\eta\downarrow 0 }\langle\rho_\eta, 1\rangle$ also belongs to $L^{p}(\mathbb{P})$, and hence (\ref{thenwearedone}) holds.
\hfill\end{proof}

\section{Proof of Theorem \ref{main}}

The  proof appeals to ideas which are are inspired by the analysis of branching processes appearing in Asmussen and Herring \cite{AH76}. First we prove almost sure convergence of $\langle\rho_\eta, f\rangle$ on log-lattice sequences and this is then upgraded to convergence along the continuous sequence $1\geq \eta>0$.

We start by appealing to the Markov-Chebyshev inequality followed by an application of Proposition \ref{Lp} to deduce that for any $\delta,\epsilon>0$, $p=p_0$ as specified in (A3),  $m=2,3,\cdots$ and bounded measurable $f:[0,1]\rightarrow[0,\infty)$
\begin{eqnarray*}
\lefteqn{\sum_{n=1}^\infty \mathbb{P}\left[\left|\langle\rho_{e^{-mn\delta}}, f\rangle- \mathbb{E} (\langle \rho_{e^{-mn\delta}}, f\rangle |\mathcal{F}_{e^{-n\delta}})\right|>\epsilon\right]}\\
&\leq &\frac{1}{\epsilon^p}\sum_{n=1}^\infty||\langle\rho_{e^{-mn\delta}}, f\rangle- \mathbb{E} (\langle \rho_{e^{-mn\delta}}, f\rangle |\mathcal{F}_{e^{-n\delta}})||_p^p\\
&\leq&\frac{\kappa_p^p}{\epsilon^p}\sum_{n=1}^\infty e^{-n\delta(p-1)/p}<\infty
\end{eqnarray*}
Together with the Borel-Cantelli Lemma this implies that, for any $\delta>0$,
\begin{equation}
\left|\langle\rho_{e^{-mn\delta}}, f\rangle- \mathbb{E} (\langle \rho_{e^{-mn\delta}}, f\rangle |\mathcal{F}_{e^{-n\delta}})\right|\rightarrow 0
\label{lattice1}
\end{equation}
$\mathbb{P}$-almost surely as $n\uparrow\infty$. 

Next appealing to the triangle inequality, we have for $s>1$  and $1\geq\eta>0$ that
\begin{eqnarray*}
\left| \langle\rho_{e^{-mn\delta}}, f\rangle-\langle\rho, f\rangle\right|&\leq &\left|\langle\rho_{e^{-mn\delta}}, f\rangle- \mathbb{E} (\langle \rho_{e^{-mn\delta}}, f\rangle |\mathcal{F}_{e^{-n\delta} })\right|\\
&&+ |\mathbb{E} (\langle \rho_{e^{-mn\delta}}, f\rangle |\mathcal{F}_{\eta}) - \langle\rho, f\rangle \Lambda(p^*)|
\end{eqnarray*}
and hence (\ref{lattice1}) and Lemma \ref{a.s.approx} imply that there exists a  natural number $m_0$ such that for each $m\geq m_0$ and $\delta>0$,
\[
\lim_{n\uparrow\infty} \langle\rho_{e^{-mn\delta}}, f\rangle   =  \langle\rho, f\rangle\Lambda(p^*)
\]
$\mathbb{P}$-almost surely.
Since we may choose $\delta>0$ in an arbitrary way (it does not depend on the value $m_0$), we may rescale $\delta$ by $m^{-1}$ and improve the above convergence to deduce
\begin{equation}
\lim_{n\uparrow\infty} \langle\rho_{e^{-n\delta}}, f\rangle   =  \langle\rho, f\rangle\Lambda(p^*) \qquad \mathbb{P}\mbox{-a.s.}
\label{improved}
\end{equation}
for any $\delta>0$.

To remove the assumption that convergence holds only along log-lattice sequences, fix $\delta>0$ and take $t\in(n\delta, (n+1)\delta)$ for $n=1,2,\cdots$.  Let us  extend the domain of $f$ to $[0,\infty)$ by defining $f(x)=0$ for $x>1$. Amongst those fragments in the stopping line $\mathbf{X}_{e^{-t}}$ are  fragments which are also to be found in $\mathbf{X}_{e^{-n\delta}}$.
It follows that 
\begin{eqnarray*}
\langle\rho_{e^{-t}}, f\rangle
&=& \sum_{
j\in\mathcal{J}^{\rm c}_{e^{-t}, 1}
\cap
\mathcal{J}^{\rm c}_{e^{-n\delta}, t/n\delta}
} X^{1+p^*}_{e^{-t}, j}f(X_{e^{-t},j }/e^{-t})\\
&&\hspace{2cm}+ 
\sum_{
j\in\mathcal{J}^{\rm c}_{e^{-t}, 1}
\backslash
\mathcal{J}^{\rm c}_{e^{-n\delta}, t/n\delta}
}  X^{1+p^*}_{e^{-t}, j}f(X_{e^{-t},j }/e^{-t})\\
&\geq& \sum_{j\in\mathcal{J}^{\rm c}_{e^{-n\delta}, t/n\delta}} X^{1+p^*}_{e^{-n\delta}, j} f(X_{e^{-n\delta},j }/e^{-t})\\
&=& \sum_{j} X^{1+p^*}_{e^{-n\delta}, j} f(X_{e^{-n\delta},j }/e^{-t})\\
&=&
\sum_j X^{1+p^*}_{e^{-n\delta}, j} f( e^{t-n\delta} X_{e^{-n\delta},j }/e^{-n\delta} ).
\end{eqnarray*}
If we assume in addition that $f$ is continuous and compactly supported (and therefore uniformly continuous) then for any given $\varepsilon>0$ we have that there exists a $\delta_0>0$ such that whenever $\delta<\delta_0$,  for all $x\in\mbox{supp}f$, $f(x e^{t-n\delta} ) \geq f(x) -\varepsilon$. In that case we have with the help of (\ref{improved}) that
\[
\liminf_{t\uparrow\infty}\langle\rho_{e^{-t}}, f\rangle\geq \liminf_{n\uparrow\infty}\langle\rho_{e^{-n\delta}}, f\rangle - \varepsilon\langle\rho_{e^{-n\delta}}, 1\rangle 
 = (\langle\rho, f\rangle-\varepsilon) \Lambda(p^*) 
\]
$\mathbb{P}$-almost surely for all $\delta<\delta_0$.
As $\varepsilon>0$ can be chosen arbitrarily small, we have that for all continuous $f:[0,1]\rightarrow[0,\infty)$ which are compactly supported,
\[
\liminf_{t\uparrow\infty}\langle\rho_{e^{-t}}, f\rangle\geq \langle\rho, f\rangle\Lambda(p^*) 
\]
$\mathbb{P}$-almost surely.
Next, given any continuous bounded $f:[0,1]\rightarrow[0,\infty)$, suppose that $\{f_n:n=1,2,\cdots\}$ is a sequence of continuous and compactly supported positive functions such that $f_n\uparrow f$ in the pointwise sense.
It follows that 
\begin{eqnarray}
\liminf_{t\uparrow\infty}\langle\rho_{e^{-t}}, f\rangle&=&\lim_{k\uparrow\infty}\liminf_{t\uparrow\infty}\langle\rho_{e^{-t}}, f\rangle\nonumber\\
&\geq & \lim_{k\uparrow\infty}\liminf_{t\uparrow\infty}\langle\rho_{e^{-t}}, f_k\rangle\nonumber\\
&\geq & \lim_{k\uparrow\infty}\langle\rho, f_k\rangle\Lambda(p^*)\nonumber\\
&=&\langle\rho, f\rangle\Lambda(p^*)
\label{liminf}
\end{eqnarray}
$\mathbb{P}$-almost surely, 
where the final equality follows from the monotone convergence theorem taking account of the expression in (\ref{usemct}). If $f:[0,1]\rightarrow [0,\infty)$ is now a  positive, bounded and measurable function, then it can be approximated from below by an increasing sequence of positive continuous functions and the computations in (\ref{liminf}) go through verbatim. 
On the other hand, for such an $f$, assuming without loss of generality that it is uniformly bounded by $1$,  we have, recalling that $\{\langle\rho_{e^{-t}},1\rangle: t\geq 0\}$ is a martingale with almost sure limit $\Lambda(p^*)$,
\begin{eqnarray*}
\limsup_{t\uparrow\infty}\langle\rho_{e^{-t}}, f\rangle 
&\leq& \lim_{t\uparrow\infty}\langle\rho_{e^{-t}}, 1\rangle - \liminf_{t\uparrow\infty} \langle\rho_{e^{-t}}, (1-f)\rangle\\
&\leq&[1- \langle\rho, (1-f)\rangle]\Lambda(p^*)\\
&=& \langle\rho,f\rangle \Lambda(p^*)
\end{eqnarray*}
$\mathbb{P}$-almost surely,
where in the last inequality we have used (\ref{liminf}) and in the final equality we have used that $\langle\rho, 1\rangle=1$.
This completes the proof of Theorem \ref{main}. \hfill$\square$

\end{document}